\documentclass[12pt]{article}
\usepackage{amssymb,amsmath, amsthm}
\usepackage{color,xcolor}
\newcommand{\eb}{\begin{equation}}
\newcommand{\ee}{\end{equation}}
\newcommand{\ebx}{\begin{equation*}}
\newcommand{\eex}{\end{equation*}}
\newtheorem{lemma}{Lemma}[section]
\newtheorem{proposition}[lemma]{Proposition}
\newtheorem{theorem}[lemma]{Theorem}

\newtheorem{definition}[lemma]{Definition}

\allowdisplaybreaks

\makeatletter
\renewcommand*\env@matrix[1][*\c@MaxMatrixCols c]{%
  \hskip -\arraycolsep
  \let\@ifnextchar\new@ifnextchar
  \array{#1}}
\makeatother

\begin{document}

\title{Local orthogonal maps and rigidity of holomorphic mappings between real hyperquadrics}

\author{Yun Gao\footnote{School of Mathematical Sciences, Shanghai Jiaotong University, Shanghai, People's Republic of China. \textbf{Email:}~gaoyunmath@sjtu.edu.cn}, Sui-Chung Ng\footnote{School of Mathematical Sciences, Shanghai Key Laboratory of PMMP, East China Normal University, Shanghai, People's Republic of China. \textbf{Email:}~scng@math.ecnu.edu.cn}}

\maketitle

\begin{abstract}
We introduced a new coordinate-free approach to study the Cauchy-Riemann (CR) maps between the real hyperquadrics in the complex projective space. The central theme is based on a notion of orthogonality on the projective space induced by the Hermitian structure defining the hyperquadrics. There are various kinds of special linear subspaces associated to this orthogonality which are well respected by the relevant CR maps and this is where the rigidities come from. Our method allows us to generalize a number of well-known rigidity theorems for the CR mappings between real hyperquadrics with much simpler arguments.
\end{abstract}

%\noindent\textbf{Acknowledgement.}
%The author was partially supported by the NSFC

\section{Introduction}

The real hyperquadrics in the complex projective space plays a special role in the field of CR (Cauchy-Riemann) Geometry. On the one hand, they are the among simplest CR-manifolds (e.g. the boundaries of the unit balls) which can serve as model spaces on which one can formulate or verify various statements or theories. On the other hand, to CR geometry they play a role like what the Euclidean spaces to Riemannian Geometry and what the projective spaces to Algebraic Geometry. For instance, there is the well-known problem about which CR manifolds can be embedded into these hyperquadrics. We refer the reader to~\cite{BEH, da, Fo, HJ, Za1} (and the references therein) for the related works in this area.

The traditional approach to the study is based on Chern-Moser's normal form theory, in which the central theme is that one can choose good coordinates such that the CR manifolds and the relevant holomorphic maps take certain normal forms. This is a powerful method which has been used to solve many problems but it usually requires formidable calculation. Nevertheless, we observed that when the CR manifolds being concerned are real hyperquadrics, there are a certain type of orthogonality and a number of related objects such as null spaces and orthogonal complements which interact well with the CR maps. Our approach to the study of real hyperquadrics is to work on these geometric objects directly. While our method might not be able to produce as much fine detail of the relevant maps as the traditional normal form theory, it has the advantages of being coordinate free and geometrically more transparent. It is especially suited for obtaining rigidities and studying general behaviors. In addition, the arguments involved are much easier. 

Let $r,s,t\in\mathbb N$ and denote by $\mathbb C^{r,s,t}$ be the Euclidean space equipped with the standard (possibly degenerate) Hermitian bilinear form whose eigenvalues are $+1$, $-1$ and $0$ with multiplicities $r$, $s$ and $t$ respectively. Consider its projectivization $\mathbb P^{r,s,t}:=\mathbb P\mathbb C^{r,s,t}$. The notions of positive, negative and null points are well defined on $\mathbb P^{r,s,t}$. Among these, the set of positive points, denoted by $\mathbb B^{r,s,t}\subset\mathbb P^{r,s,t}$, is called a \textit{generalized ball} and its boundary $\partial\mathbb B^{r,s,t}$, which consists of the set of null points, is a CR hypersurface in $\mathbb P^{r,s,t}$ defined by a real quadratic equation. The name ``generalized balls" comes from the fact that $\mathbb B^{1,s,0}$ is just the ordinary $s$-dimensional complex unit ball embedded in $\mathbb P^s$. 

Suppose $U\subset\mathbb P^{r,s,t}$ is a connected open set, $U\cap\partial\mathbb B^{r,s,t}\neq\varnothing$ and $f:U\rightarrow\mathbb P^{r',s',t'}$ is a holomorphic map such that $f(U\cap\partial\mathbb B^{r,s,t})\subset\partial\mathbb B^{r',s',t'}$. A classical problem in CR geometry is under what conditions $f$ is rigid in the sense that it comes from some ``standard" map between the projective spaces, e.g. from a linear map. Our starting point is the observation that $f$ preserves the orthogonality induced by the Hermitian bilinear form on $\mathbb C^{r,s,t}$ which remains defined on $\mathbb P^{r,s,t}$ (Proposition~\ref{equiv1}).  This leads to the definition of \textit{local orthogonal maps} (see Definition~\ref{orthogonal def} and the remark there regarding the relation to CR maps between hyperquadrics). In what follows, we call $f$ \textit{null} if $f(U)\subset\partial\mathbb B^{r',s',t'}$. This is a kind of triviality in the current setting analogous to constant maps. We call $f$ \textit{quasi-linear} (resp. \textit{quasi-standard}) if it is in some sense a ``direct sum" of two parts, of which one comes from a linear map (resp. linear isometry) from $\mathbb C^{r,s,t}$ to $\mathbb C^{r',s',t'}$ and the other one is null. The precise definitions are given in Definitions~\ref{standard def} and~\ref{quasi def} and for the moment we use a simple example of quasi-standard maps for illustration: 

\noindent\textbf{Example.} \textit{We split the coordinates of $\mathbb C^{r,s}$ as $(z^+,z^-)$ corresponding to the positive and negative eigenvalues of the standard Hermitian form. Let $\phi$, $\psi$ and $\chi$ be homogeneous polynomials function on $\mathbb C^{r,s}$ such that $\deg(\psi)=\deg(\chi)=\deg(\phi)+1$. Then a rational map from $\mathbb P^{r,s}$ to $\mathbb P^{r+1,s+1,1}$ of the form $[z^+,z^-]\mapsto [\phi z^+, \psi,\phi z^-,\psi,\chi]$ is a quasi-standard orthogonal map, in which $[\phi z^+,0,\phi z^-,0,0]\sim [z^+,0,z^-,0,0]$ is the standard (i.e. isometry) part and $[0,\psi,0,\psi,\chi]$ is the null part.}

We are now ready to summarize our main results for local orthogonal maps:

\begin{theorem}\label{intro thm 1}
Let $U\subset\mathbb P^{r,s,t}$ be a connected open set such that $U\cap\partial\mathbb B^{r,s,t}\neq\varnothing$ and $f:U\rightarrow\mathbb P^{r',s',t'}$ be a local orthogonal map. Then $f$ is either null or quasi-linear if one of the conditions below is satisfied:
$$
\begingroup
\renewcommand*{\arraystretch}{1.5}
\begin{matrix}[ccccl]
(i)& r,s\geq 2& \textrm{and} &\min\{r',s'\}\leq\min\{r,s\};&{\rm (Theorem~\ref{same})}\\
(ii)&t=0& \textrm{and} &\min\{r',s'\}\leq2\min\{r,s\}-2;&{\rm (Theorem~\ref{main})}\\
(iii)&t=0& \textrm{and} & r'+s'\leq 2\dim(\mathbb P^{r,s})-1. &{\rm (Theorem~\ref{double dim thm})}
\end{matrix}
\endgroup
$$
In addition, $f$ is quasi-standard if it maps a positive point to a positive point under any of the conditions above, or
$$
\begin{matrix}[ccccl]
(iv)& r,s\geq 2,\,\, r=r' &\textrm{and}& f(U\cap\mathbb B^{r,s,t})\subset\mathbb B^{r',s',t'}.&{(\rm Theorem~\ref{same2})}
\end{matrix}
$$
\end{theorem}

From Theorem~\ref{intro thm 1}, we can deduce and generalize a number of well-known rigidity theorems for the holomorphic maps between real hyperquadrics, including those of Baouendi-Huang~\cite{BH} (from $(i)$ and $(iv)$); Baouendi-Ebenfelt-Huang~\cite{BEH} (from $(ii)$); Faran~\cite{faran} (from $(iii)$); and Xiao-Yuan~\cite{XY} (from $(iii)$). The detail of the relation of these results to ours is in Section~\ref{proper maps}.

Here we briefly explain the basic philosophy behind our method, which is essentially the interplay between orthogonality and linear subspaces. First of all, orthogonality can be used to obtain important information about the images of the linear subspaces. For one thing, the maximal null spaces are easily seen to be respected by an orthogonal map $f$. Furthermore, since maximal null spaces constitute in some sense ``a distinguished subset" in the Grassmannian, we can then bound the dimension of the linear spans of the images of \textit{all} subspaces of the same dimension. For another thing, if we take a pair of orthogonal subspaces, then the linear spans of their images under an orthogonal map $f$ remain orthogonal. However, the sum of the dimensions of two orthogonal subspaces is bounded for a given target space and so this will again control the linear spans of the images of linear subspaces. Once such dimension estimates are obtained, by making appropriate hypotheses on the parameters defining the Hermitian forms and with the help of a technical proposition relating the images of the  subspaces in different dimensions (Proposition~\ref{faran type}), we can deduce the desired rigidity.

%-----------------------------------------------------------------------------------------

\section{Definitions and basic properties}\label{orthogonal section}

%For any non-negative integers $r,s,t$, consider the standard 
%Hermitian form of signature $(r;s;t)$ on $\mathbb C^{r+s+t}$ whose eigenvalues are $1$, $-1$ and $0$ of multiplicities $r$, $s$ and $t$ respectively. 
%The Hermitian form can be represented by the matrix
% $H_{r,s,t}=\begin{pmatrix}
% I_r&0&0\\
% 0&-I_s&0\\
% 0&0&0_t
% \end{pmatrix}$ with a suitable choice of basis. 
Let $r,s,t\in\mathbb N$ and $n:=r+s+t>0$. Define the (possibly degenerate) indefinite inner product of signature $(r;s;t)$ on $\mathbb C^{n}$:
$$
	\langle z, w\rangle_{r,s,t}
	=z_1\bar w_1+\cdots+z_r\bar w_r-z_{r+1}\bar w_{r+1}-\cdots - z_{r+s}\bar w_{r+s},
$$
where $z=(z_1,\ldots,z_{n})$ and $w=(w_1,\ldots,w_{n})$. We also define the indefinite norm  $\|z\|^2_{r,s,t}=\langle z, z\rangle_{r,s,t}$. Then, for any $z\in \mathbb C^{r,s,t}$, we call it a \textit{positive point} if $\|z\|^2_{r,s,t}>0$; a \textit{negative point} if  $\|z\|^2_{r,s,t}<0$ and a \textit{null point} if  $\|z\|^2_{r,s,t}=0$. If $\langle z, w\rangle_{r,s,t}=0$, we say that $z$ is orthogonal to $w$ and write $z\perp w$. In addition, the \textit{orthogonal complement} of $z$ is defined as
$$z^{\perp}=\{w\in  \mathbb C^{r,s,t} \mid \langle z, w\rangle_{r,s,t}=0\}.$$

We denote by $\mathbb C^{r,s,t}$ the $\mathbb C^{n}$ with the Hermitian inner product defined above and by 
$\mathbb P^{r,s,t}:=\mathbb P\mathbb C^{r,s,t}$ its projectivization.  We write $\mathbb C^{r,s}$ and $\mathbb P^{r,s}$ instead of $\mathbb C^{r,s,0}$ and $\mathbb P^{r,s,0}$.
A biholomorphism on $\mathbb P^{r,s,t}$ induced by a linear isometry of $\mathbb C^{r,s,t}$ is said to be an automorphism of  $\mathbb P^{r,s,t}$. 
%Let  $M(p;q;\mathbb C) $ be the set of $p\times q$ complex matrices. For any invertible $\sigma\in M_{r,s,t}:=\{A \in M(n,n;\mathbb C) \mid A^HH_{r,s,t}A=H_{r,s,t}\}$, $\sigma$ can define an automorphism of $\mathbb P^{r,s,t}$. Here $A^H$ denotes the Hermitian transpose of $A$.

Although the norm $\|\cdot\|^2_{r,s,t}$ of course does not descend to $\mathbb P^{r,s,t}$, the positivity, negativity or nullity of a line (1-dimensional subspace) remains well defined and thus we can talk about positive points, negative points and null points on  $\mathbb P^{r,s,t}$. Furthermore, the orthogonality of two points on $\mathbb P^{r,s,t}$ and hence the notion of orthogonal complement also make sense on $\mathbb P^{r,s,t}$.

More generally, let $V$ be a complex vector space equipped with a Hermitian inner product (possibly degenerate or indefinite) $H_V$ of signature $(r;s;t)$, where $\dim(V)=r+s+t$. Let $\mathbb PV$ be its projectivization. The notion of positivity, negativity, nullity and orthogonality can be similarly defined on $\mathbb PV$. In addition, any linear isometry $F:\mathbb C^{r,s,t}\rightarrow V$ induces a biholomorphic map $\tilde F:\mathbb P^{r,s,t}\rightarrow\mathbb PV$ preserving all these notions. Sufficient for our purpose, we  can simply identify any such projective space $\mathbb PV$ with $\mathbb P^{r,s,t}$ through any such biholomorphism and we write $\mathbb PV\cong\mathbb P^{r,s,t}$ for such identification.

Now let $H$ be a complex linear subspace in $\mathbb C^{r,s,t}$ and the restriction of  $\langle\cdot,\cdot\rangle_{r,s,t}$ on $H$ has the signature $(a; b; c)$. Obviously, we have $0\le a\le r$, $0\le b\le s$, $0\le c\le \min\{r-a,s-b\}+t$ and $a+b+c=\dim(H)$.  Then $\mathbb PH \cong \mathbb P^{a,b,c}$. We call $\mathbb P H$ an\textit{ $(a,b,c)$-subspace} of $\mathbb P^{r,s,t}$. Usually, we denote an $(a,b,c)$-subspace by $H^{a,b,c}$.

If $a=b=0$, $\mathbb P H$ is called a \textit{null space}. Similarly, it is called a \textit{positive space} (resp. \textit{negative space}) if $b=c=0$ (resp. $a=c=0$).  We will also use the terms \textit{null $k$-plane, positive $k$-plane} and \textit{negative $k$-plane} when the $\dim(\mathbb PH)=k$.
Obviously, the maximum dimension of the null spaces in $\mathbb P^{r,s,t}$ is $\min\{r,s\}+t-1$.  The null spaces with the maximal dimension are called the \textit{maximal null spaces}.

We now recall the definition of type-I irreducible bounded symmetric domain and with that we can   give some useful parametrizations for the positive, negative and null spaces in $\mathbb P^{r,s,t}$.

\begin{definition}\label{BSD}
Let $M_{r,s}$ be the set of $r\times s$ complex matrices. The type-I irreducible bounded symmetric domain $\Omega_{r,s}$ is the domain in $M_{r,s}\cong\mathbb C^{rs}$ defined by $\Omega_{r,s}=\{A\in M_{r,s}: I-AA^H>0 \}$, where $A^H$ denotes the Hermitian transpose of $A$. 
\end{definition}

In what follows, for a point $[z]=[z_1,\ldots,z_n ]\in \mathbb P^{r,s,t}$, we split the homogeneous coordinates as $[z]=[z^+,z^-,z^0]$,  where $z^+=[z_1,\cdots, z_r]$, $z^-=[z_{r+1},\cdots, z_{r+s}]$, $z^0=[z_{r+s+1},\cdots, z_n]$. If $t=0$, $[z]$ is split as $[z]=[z^+,z^-]$.

Let  $A\in M_{r,s}$ and $B\in M_{r,t}$. Consider the $(r-1)$-plane defined by
$$
H_{A,B}=\{[z^+,z^-,z^0]\in \mathbb P^{r,s,t} \mid z^-=z^+A\,\,\textrm{and}\,\, z^0=z^+B\}\cong \mathbb P^{r-1}\subset \mathbb P^{r,s,t}.
$$
Using the definition of $H_{A,B}$, we see that $M_{r,s} \times M_{r,t}$ can be identified naturally as an open subset  of the Grassmannian $\mathbb G(r-1, \mathbb P^{r,s,t})$  which is the set of $(r-1)$-planes in $\mathbb P^{r,s,t}$. (Note that the dimensions of $M_{r,s}\times M_{r,t}$  and  $\mathbb G(r-1, \mathbb P^{r,s,t})$ are the same.)

The $(r-1)$-plane $H_{A,B}$ is a positive subspace if and only if $\|z^+\|^2>\|z^-\|^2=\|z^+A\|^2$ for all $z^+\in\mathbb P^{r-1}$, which is in turn equivalent to $A\in \Omega_{r,s}$. 
Similarly, when $r\leq s$, we see that $H_{A,B}$ is a null $(r-1)$-plane if and only if $AA^H=I$. We know that latter equation defines precisely the Shilov boundary of $\Omega_{r,s}$ in $M_{r,s}\cong\mathbb C^{rs}$, denoted by $S(\Omega_{r,s})$. 
To summarize, we have the following (cf.~\cite{NZ}):

\begin{proposition}\label{parametrization}
There is an open embedding of $M_{r,s}\times M_{r,t}$ into $\mathbb G(r-1,\mathbb P^{r,s,t})$ such that its restriction to $\Omega_{r,s}\times M_{r,t}$ gives a parametrization of all positive $(r-1)$-planes in $\mathbb P^{r,s,t}$. In addition, when $r\leq s$, the restriction to $S(\Omega_{r,s})\times M_{r,t}$  also gives a parametrization of the null $(r-1)$-planes in $\mathbb P^{r,s,t}$.
\end{proposition}

Many analyses on the mapping problems between CR manifolds begin with the fact that the associated \textit{Segre varieties} are well respected by CR maps. The study of Segre varieties has a very important role in many problems like reflection principle and algebracity~\cite{Za2}. For real hyperquadrics on complex projective space, we note that their Segre varieties are just the orthogonal complements (to points) with respect to the orthogonality described previously. Motivated from this, we thus give the following definition.

\begin{definition}\label{orthogonal def}
Let $U \subset \mathbb P^{r,s,t}$ be  a connected open set containing a null point. We call a holomorphic map $F: U\rightarrow\mathbb P^{r',s',t'}$ \textbf{orthogonal} if $F(p)\perp F(q)$ for any $p, q\in U$ such that $p\perp q$; \textbf{sign-preserving} if $F$ maps positive points to positive points and negative points to negative points.
%; a \textbf{non-negative map} if $F(U)$ does not contain any negative point; a \textbf{non-positive map} if $F(U)$ does not contain any positive point.
We also simply call $F$ a local orthogonal map or local sign-preserving map from $ \mathbb P^{r,s,t}$ to $\mathbb P^{r',s',t'}$ in such cases.
\end{definition}

%\noindent{\bf Remark 1.} It follows easily from continuity that a sign preserving map also maps null points to null points.

\noindent{\bf Remark.} For the sake of convenience, in our definition we require the domain of definition $U$ of a local orthogonal map $f$ to intersect $\partial\mathbb B^{r,s,t}$ so that the orthogonality or sign-preserving condition will not become vacuous. (Otherwise, it could happen that  $p^{\perp}\cap U=\varnothing$ for any point $p\in U$.) With this requirement, Proposition~\ref{null} and Proposition~\ref{equiv1} below will imply that orthogonal maps and holomorphic maps between real hyperquadrics are the same set of mappings.
However, we remark that in principle one can define orthogonal maps more generally without requiring $U\cap\partial\mathbb B^{r,s,t}\neq\varnothing$ and some of our results (e.g. Theorem~\ref{double dim thm}) will still be valid. In this sense, a local orthogonal map is a slightly more general object.

\noindent{\bf Convention.} \textit{In what follows, if we say that $F$ maps lines to lines (resp. $k$-planes to $k'$-planes), we mean $F$ maps the intersection of any line (resp. $k$-plane) and $U$  into a line (resp. $k'$-plane).}

Local orthogonal map preserves null spaces, as demonstrated below.

\begin{proposition}\label{null}
Let $F$ be a local orthogonal map from $ \mathbb P^{r,s,t}$ to $\mathbb P^{r',s',t'}$. Then $F$ maps null spaces to null spaces (in particular, null points to null points).
\end{proposition}
\begin{proof}
For any two points $\alpha, \,\beta$ in a null space in $\mathbb P^{r,s,t}$, we have $\langle \alpha, \,\alpha\rangle_{r,s,t}=\langle \beta, \,\beta\rangle_{r,s,t}=\langle \alpha, \,\beta\rangle_{r,s,t}=0$.
Since $F$ is a local orthogonal map, we get
$$\langle F(\alpha), \,F(\alpha)\rangle_{r',s',t'}=\langle F(\beta), \,
F(\beta)\rangle_{r',s',t'}=\langle F(\alpha), \,F(\beta)\rangle_{r',s',t'}=0.$$  
So the linear span of the image of a null space is a null space.
\end{proof}

There are two kinds of null points in $\mathbb P^{r,s,t}$ when $r,s,t>0$. The first kind of null points  is the null points $\alpha$ satisfying $\langle \alpha,\beta \rangle_{r,s,t}=0$, for any $\beta\in \mathbb P^{r,s,t}$. We call these null points \textit {special null points} and call other null points \textit{ordinary null points.} It is easy to see from the defining equations for special and ordinary null points that a \textit{general} null point is ordinary and for $r,s>0$, whenever an open set contains a null point, it must contain an ordinary null point.

%{\color{red} 
%\begin{proposition}\label{snv}
%Let $U\subset \mathbb P^{r,s,t}$ be an open set containing a special null point. 
%If an orthogonal map $F: U \to \mathbb P^{r',s',t'}$ does not preserve special null points, i.e. $F$ maps a special null point to an ordinary null point, then $F$ maps $\mathbb P^{r,s,t}$ into a hyperplane in $\mathbb P^{r',s',t'}$.
%\end{proposition}
%\begin{proof}
%Let $\alpha$ be a special null point. Then $\alpha^{\perp}=\mathbb P^{r,s,t}$. If $F(\alpha)$ is an ordinary null point, then $\alpha^{\perp}$  is an $(r'-1,s'-1,t'+1)$-subspace in $\mathbb P^{r',s',t'}$.  Hence $F$  maps $ \mathbb P^{r,s,t}$ into this $(r'-1,s'-1,t'+1)$-subspace  which is a hyperplane in $\mathbb P^{r',s',t'}$. 
%\end{proof}
%}

%From the theorem above, the image of a special null point can tell more information about the local orthogonal map than the image of a general point. When we consider a local orthogonal map from $\mathbb P^{r,s,t}$ to $\mathbb P^{r',s'}$, usually, we assume that the definition domain doesn't contain a special null point. Otherwise, instead, we can study a  local orthogonal map from $\mathbb P^{r,s,t}$ to $\mathbb P^{r'-1,s'-1,1}$.

Both orthogonal maps and sign-preserving maps map null points to null points. Conversely, the following proposition in particular implies that a sign-preserving map is locally an orthogonal map. 

\begin{proposition}\label{equiv1}
Let $r,s>0$ and $F:U\subset\mathbb P^{r,s,t}\rightarrow\mathbb P^{r',s',t'}$ be a holomorphic map, where $U$ is an open set containing a null point. If $F$ maps null points to null points, then there exists an open set $V\subset U$, such that $F:V\rightarrow\mathbb P^{r',s',t'}$ is an orthogonal map. 
\end{proposition}

\begin{proof}

Let $n=r+s+t$, $n'=r'+s'+t'$ and write the homogeneous coordinates of $\mathbb P^{r,s,t}$ and $\mathbb P^{r',s',t'}$ as $[z_1,\ldots, z_n]$ and $[w_1,\ldots, w_{n'}]$ respectively. Since $r,s>0$, we know that $U$ must contain an ordinary null point and hence by shrinking $U$ if necessary, we may assume without loss of generality that $U$ is contained in the open set $U_1\cong\mathbb C^{n-1}\subset\mathbb P^{r,s,t}$ defined by $z_1\neq 0$ 
and $F(U)$ is contained in the open set of $U'_1\cong\mathbb C^{n'-1}\subset\mathbb P^{r',s',t'}$ defined by $w_1\neq 0$.  We write the standard inhomogeneous coordinates in $U_1$ as $(\zeta_2,\ldots,\zeta_n)$, where $\zeta_j=z_j/z_1$ and similarly write $(\eta_2,\ldots,\eta_{n'})$ for $U'_1$, where $\eta_\ell=w_\ell/w_1$.
In terms of these coordinates, the null points in $U_1$ and $U'_1$ are respectively given by the equations 
$$
1+\sum^r_{j=2}|\zeta_j|^2-\sum^{r+s}_{j=r+1}|\zeta_j|^2=0 \,\,\,\,
\textrm{and}\,\,\,\,
1+\sum^{r'}_{\ell=2}|\eta_\ell|^2-\sum^{r'+s'}_{\ell=r'+1}|\eta_\ell|^2=0.
$$ 
(Strictly speaking, these equations are for the cases where $r,r'\geq 2$ but other cases can be handled in a similar fashion.)

Using the inhomogeneous coordinates above, write $F=(F_2,\cdots,F_{n'})$. If $F$ maps null points to null points, then there exist a connected open set $V\subset U$ containing a null point, $k\in\mathbb N^+$ and a real analytic function $\rho$ on $V$ such that
\begin{equation}\label{rho eq}
1+\sum^{r'}_{\ell=2}|F_\ell|^2-\sum^{r'+s'}_{\ell=r'+1}|F_\ell|^2=\left(1+\sum^r_{j=2}|\zeta_j|^2-\sum^{r+s}_{j=r+1}|\zeta_j|^2\right)^k\rho
\end{equation}
holds on $V$.
Hence, by shrinking $V$ if necessary, we can polarize the equation, i.e. for any $\zeta,\xi\in V$, we have
$$
1+\sum^{r'}_{\ell=2}F_\ell(\zeta)\overline{F_\ell(\xi)}-\sum^{r'+s'}_{\ell=r'+1}F_\ell(\zeta)\overline{F_\ell(\xi)}=\left(1+\sum^r_{j=2}\zeta_j\overline{\xi_j}-\sum^{r+s}_{j=r+1}\zeta_j\overline{\xi_j}\right)^k\rho (\zeta,\bar\xi),
$$
where $\zeta=(\zeta_2,\ldots,\zeta_n)$, $\xi=(\xi_2,\ldots,\xi_n)$.

Now, $\zeta$ as a point in $\mathbb P^{r,s,t}$ has homogeneous coordinates $[1,\zeta_2,\ldots,\zeta_n]$ and thus $\zeta^\perp\cap V$ consists precisely the points $\xi=[1,\xi_2,\ldots,\xi_n]\in V$ satisfying $\displaystyle 1+\sum^r_{j=2}\zeta_j\overline{\xi_j}-\sum^{r+s}_{j=r+1}\zeta_j\overline{\xi_j}=0$. With a similar consideration for $F(\zeta)^\perp$, we conclude from the above equation that $F(\zeta^\perp)\subset (F(\zeta))^\perp$ for $\zeta\in V$.
\end{proof}

%A local orthogonal map must map maximal null spaces into maximal null spaces. 
It has been known that for any local proper holomorphic map $f$ between two generalized balls (excluding the unit balls), there exist $k,k'\in\mathbb N^+$ such that $f$ maps $k$-planes to  $k'$-planes (\cite{Ng1} Proposition 4.1 therein).  The following proposition is a generalization of this result to local orthogonal maps, which is very crucial in our study.

\begin{proposition}\label{boundary}
If $F$ is a local orthogonal map from $\mathbb P^{r,s,t}$ to $\mathbb P^{r',s',t'}$, then $F$ maps $(\min\{r,s\}-1)$-planes to $(\min\{r',s'\}+t'-1)$-planes.

\end{proposition}
\begin{proof}
By symmetry, it suffices to prove the case for $r\leq s$ and $r'\leq s'$. The case $r=1$ is trivial and so we let $r\geq 2$. 
Recall from Proposition~\ref{parametrization} that the null $(r-1)$-planes in $\mathbb P^{r,s,t}$ can be parametrized by $S(\Omega_{r,s})\times M_{r,t}\subset \mathbb G(r-1, \mathbb P^{r,s,t})$, where $S(\Omega_{r,s})$ is the Shilov boundary of the type-I bounded symmetric domain $\Omega_{r,s}$.
Now from Proposition \ref{null}, the map $F$ maps null spaces to null spaces and thus the image of every null $(r-1)$-plane in $\mathbb P^{r,s,t}$ is contained in a maximal null space, which is an $(r'+t'-1)$-plane in $\mathbb P^{r',s',t'}$. 

Note that for any point $p\in \mathbb G(r-1,\mathbb P^{r,s,t})$, there  exists a set of local holomorphic functions in a neighborhood $\mathcal U\ni p$ such that they vanish at a point $q\in\mathcal U$ precisely when the image under $F$ of the $(r-1)$-plane given by $q$ is contained in an $(r'+t'-1)$-plane. (For instance, one can consider a set of determinants given by the Taylor coefficients of $F$. For more detail, see~\cite{NZ}, Proof of Theorem 1.1 therein.) From the previous paragraph, we know that when choose $p\in S(\Omega_{r,s})\times M_{r,t}$, these functions vanish at $\mathcal U\cap (S(\Omega_{r,s})\times M_{r,t})$. The Shilov boundary is the distinguished boundary of $\Omega_{r,s}$ and from this we deduce that these functions actually vanish identically on the open set $\mathcal U$ (see~\cite{Ng2} or~\cite{NZ}). Thus, $F$ maps $(r-1)$-planes to $(r'+t'-1)$-planes.
\end{proof}

%------------------------------------------------------------------------------------------------------

\section{Rigidity of local orthogonal maps}\label{rigidity section}

In order to simplify the presentation, we will give a couple of definitions. In what follows, when we say that a local holomorphic map $F$ between projective spaces is \textit{linear}, we mean $F$ extends to a linear rational map.

\begin{definition}\label{standard def}
Let $F$ be a local holomorphic map from $\mathbb P^{r,s,t}$ to $\mathbb P^{r',s',t'}$. We call $F$ \textbf{standard} if it is linear and comes from a linear isometry from $\mathbb C^{r,s,t}$ into $\mathbb C^{r',s',t'}$. We call $F$ \textbf{null} if its image is contained in  a null space in $\mathbb P^{r',s',t'}$.
\end{definition}

For any non-trivial orthogonal direct sum decomposition $\mathbb C^{r,s,t}=A\oplus B$, there are two canonical projections $\pi_A:\mathbb P^{r,s,t}\dasharrow\mathbb PA$ and $\pi_B:\mathbb P^{r,s,t}\dasharrow\mathbb PB$ (as rational maps). Using these, we make the following definition.

\begin{definition}\label{quasi def}
Let $F$ be a local holomorphic map from $\mathbb P^{r,s,t}$ to $\mathbb P^{r',s',t'}$. We call $F$ \textbf{quasi-standard} (resp. \textbf{quasi-linear}) if either $F$ is standard (resp. linear) or there exists a non-trivial orthogonal decomposition $\mathbb C^{r',s',t'}=A\oplus B$ for some subspaces $A, B$ such that $\pi_A\circ F$ is standard (resp. linear) and $\pi_B\circ F$ is null.
\end{definition}

Since the projection from $\mathbb P^{r',s',t'}$ to $\mathbb P^{r',s'}$ is also an orthogonal map (see the lemma below), therefore a local orthogonal map $F$ from $\mathbb P^{r,s,t}$ to $\mathbb P^{r',s',t'}$ naturally gives rise to a local orthogonal map from $\mathbb P^{r,s,t}$ to $\mathbb P^{r',s'}$ (unless the image lies entirely in the space of special null vectors, which is the set of indeterminacy for the projection). The same thing happens for local sign-preserving maps. This allows us to reduce the problem to the case for $t'=0$ in many situations. Moreover, the orthogonality of mappings also interacts nicely with orthogonal decompositions. The following two propositions should be evident to the reader and we omit their proofs.
 
\begin{proposition}\label{projection}
The projection $\pi: \mathbb P^{r,s,t}\dashrightarrow  \mathbb P^{r,s}$ is standard. In particular, it is sign-preserving map and orthogonal.
\end{proposition}
%\begin{proof}
%As the projection $\mathbb C^{r,s,t}\rightarrow\mathbb C^{r,s}$ preserves the inner product, the lemma %follows.
%\end{proof}

\begin{proposition}\label{projection 2}
Let $F$ be a local holomorphic map from $\mathbb P^{r,s,t}$ to $\mathbb P^{r',s',t'}$ and $\mathbb C^{r',s',t'}=
A\oplus B$ be an orthogonal decomposition such that the image of $F$ does not lie entirely in $\mathbb PA$ or $\mathbb PB$. Let $\pi_A:\mathbb P^{r',s',t'}\dasharrow\mathbb PA$ and $\pi_B:\mathbb P^{r',s',t'}\dasharrow\mathbb PB$ be the canonical projections. If $F$ and $\pi_A\circ F$ are orthogonal, then so is $\pi_B\circ F$.
\end{proposition}

We have already seen that how orthogonal maps respect various linear subspaces associated to the Hermitian structure defining the real hyperquadrics. The following proposition serves as a bridge connecting that property of orthogonal maps to their rigidity.

\begin{proposition}\label{faran type}
Let $g:U\subset\mathbb P^m\rightarrow\mathbb P^{m'}$ be a local holomorphic map which maps $\ell$-planes to $\ell'$-planes.

$(i)$ If $\ell'\leq\ell$, then either the image of $g$ is contained in an $\ell'$-plane or $g$ extends to a linear rational map. In addition, the image of $g$ must be contained in an $\ell'$-plane  if $\ell'\leq\ell-1$;

$(ii)$ If $\ell\leq \ell'\leq 2\ell-1$, then $g$ maps $(\ell+k)$-planes to $(\ell'+k)$-planes for $k\geq 0$.
\end{proposition}

\begin{proof}
For $(i)$, if $\ell'\leq\ell$ and the image of $g$ is not contained in an $\ell'$-plane, by considering a general pair of $\ell$-planes such that their intersection is an $(\ell-1)$-plane, we see $g$ maps $(\ell-1)$-planes to $(\ell'-1)$-planes. Inductively, we deduce that $g$ maps lines to lines and hence $g$ is linear. For a proof of the last fact, see~\cite{Ng3} (Lemma 4.1 therein). If in addition, $\ell'\leq\ell-1$ and the image of $g$ is not contained in an $\ell'$-plane, we conclude again that $g$ is linear. We can then find an $(\ell'+1)$-plane $E$ with $E\cap U\neq\varnothing$ such that $g(E\cap U)$ is an open piece of an $(\ell'+1)$ plane, but $\dim(E)=\ell'+1\leq\ell$, this contradicts the hypothesis that $g$ maps $\ell$-planes to $\ell'$-planes.

For $(ii)$, it follows from a hyperplane restriction theorem of holomorphic mappings (see~\cite{GN}, Theorem 1.2). For an independent proof, see also~\cite{Ga}, Lemma 5.1 therein.
\end{proof}

%\begin{proof}
%Let $U$ be the domain of definition of $F$ and $V:=U\setminus F^{-1}(\mathbb PA\cup\mathbb PB)$. Then $\pi_A\circ F$ and $\pi_B\circ F$ are both defined on $V$. Let $p\in V$, $q\in p^\perp\cap V$ and $p'=F(p)$, $q'=F(q)$. 

%Let $\ell_{p'}$, $\ell_{q'}$, $\ell^A_{p'}$, $\ell^A_{q'}$, $\ell^B_{p'}$ and $\ell^B_{q'}$ be the lines in $\mathbb C^{r',s',t'}$ corresponding to $p'$, $q'$, $\pi_A(p')$, $\pi_A(q')$, $\pi_B(p')$ and $\pi_B(q')$ respectively. Then $\ell_{p'}^A\perp\ell_{p'}^B$ and $\ell_{q'}^A\perp\ell_{q'}^B$ by definition, and by the orthogonality of $F$ and $\pi_A\circ F$, we also have $\ell_{p'}\perp\ell_{q'}$ and $\ell^A_{p'}\perp\ell^A_{q'}$.

%Now choose any non-zero vectors $v_{p'}\in\ell_{p'}$ and $v_{q'}\in\ell_{q'}$. Write $v_{p'}=v_{p'}^A+v_{p'}^B$ and $v_{q'}=v_{q'}^A+v_{q'}^B$ be the decomposition with respect to $\mathbb C^{r',s',t'}=A\oplus B$. Since $p', q'\not\in\mathbb PA\cup\mathbb PB$, it follows that $v_{p'}^A, v_{p'}^B, v_{q'}^A, v_{q'}^B$ are non-zero. From the previous paragraph, we know that $v_{p'}\perp v_{q'}$, $v_{p'}^A\perp v_{q'}^B$,  $v_{p'}^B\perp v_{q'}^A$,  and $v^A_{p'}\perp v^A_{q'}$. Thus, we get $v^B_{p'}\perp v^B_{q'}$ and hence $\ell^B_{p'}\perp\ell^B_{q'}$. That means $\pi_B(p')$ and $\pi_B(q')$ are orthogonal, i.e. $\pi_B\circ F$ is orthogonal.
%\end{proof}

We are now ready to prove our rigidity theorems for local orthogonal maps.

\begin{theorem}\label{same}
Let $F$ be a local orthogonal map from $\mathbb P^{r,s,t}$ to $\mathbb P^{r',s',t'}$, where $r,s\geq 2$. If
$$\min\{r',s'\}\leq \min\{r,s\},$$ then $F$ is either null or quasi-linear {\rm (}linear for $t'=0${\rm )}. If in addition $F$ preserves the sign of any single positive or negative point, then $F$ is quasi-standard {\rm (}standard for $t'=0${\rm )}.
\end{theorem}

\begin{proof}
We first consider the case $t'=0$. From Proposition~\ref{boundary}, we know that $F$ maps $(\min\{r,s\}-1)$-planes into $(\min\{r',s'\}-1)$-planes. Suppose $F$ is not linear. Since $r,s\geq 2$ and $\min\{r',s'\}\leq \min\{r,s\}$, by Proposition~\ref{faran type}$(i)$, the linear span the image of $F$, denoted by $S$, is of dimension at most $(\min\{r',s'\}-1)$. 

We claim that $S$ is a null space. If on the contrary $S$ is not null, then $S\cong\mathbb P^{a,b,c}$ for some $a,b,c$ and the dimension of the maximal null space of $\mathbb P^{a,b,c}$ is $\min\{a,b\}+c-1$, which is strictly less than $\dim(S)$. Now $F$ can be regarded as a local orthogonal map from $\mathbb P^{r,s,t}$ into $\mathbb P^{a,b,c}$ which is not linear, and thus by Propositions~\ref{boundary} and~\ref{faran type}$(i)$ again, the image of $F$ is contained in a linear subspace of dimension $\min\{a,b\}+c-1<\dim(S)$, contradicting to the fact that $S$ is the linear span of the image of $F$. So $S$ is a null space and hence $F$ is null. We have thus shown that $F$ is either linear or null. 
If in addition $F$ preserves the sign of a positive or negative point, then $F$ cannot be null and from Lemma~\ref{linear} below, we see that $F$ must be standard.

Since $\mathbb C^{r',s',t'}=\mathbb C^{r',s'}\oplus\mathbb C^{0,0,t'}$ is an orthogonal direct sum, the desired result for the general case now follows directly from the case $t'=0$ and Definition~\ref{quasi def}.
\end{proof}

\begin{lemma}\label{linear}
 Let $G:\mathbb C^{r,s,t}\rightarrow\mathbb C^{r',s',t'}$ be linear and it maps null vectors to null vectors. Then, there exists $\lambda\in\mathbb R$ such that $\langle G(u), G(v)\rangle_{r',s',t'}=\lambda \langle u,v\rangle_{r,s,t}$ for any $u,v\in\mathbb C^{r,s,t}$. Moreover, if $G$ preserves the sign of any single positive or negative vector, then $r\leq r'$, $s\leq s'$ and $\lambda>0$.
\end{lemma}
\begin{proof}
The proof is just standard linear algebra. For the detail, we refer the reader to~\cite{NZ}, Lemma 4.2 therein.
\end{proof}

We can say more when $\min\{r',s'\}<\min\{r,s\}$:

\begin{theorem}\label{less}
If $\min\{r',s'\}<\min\{r,s\}$, then any local orthogonal map $F$ from $\mathbb P^{r,s,t}$ to $\mathbb P^{r',s',t'}$ is null.
\end{theorem}

\begin{proof}
When $\min\{r,s\}=1$, the result is obvious. Now, suppose $\min\{r,s\}\geq 2$.
We start from the case $t'=0$. By Proposition~\ref{boundary}, $F$ maps $(\min\{r,s\}-1)$-planes to $(\min\{r',s'\}-1)$-planes. Then Proposition~\ref{faran type}$(i)$ says that the linear span of the image of $F$, denoted by $S$ is of dimension at most $(\min\{r',s'\}-1)$. Exactly the same argument as in Theorem~\ref{same} shows that $S$ is a null space and hence $F$ is null.

Now, when $t'>0$, we know that $\pi \circ F$ is null, where $\pi$ is the projection from $\mathbb P^{r',s',t'}$ to $\mathbb P^{r',s'}$. Hence, $F$ is also null.
\end{proof}

\begin{proposition}\label{less2}
If $r'< r$ or $s'< s$, then there is no local sign-preserving map from $\mathbb P^{r,s,t}$ to $\mathbb P^{r',s',t'}$.
\end{proposition}
\begin{proof}
By symmetry, it suffices to prove the statement for $r'<r$.
The statement is trivial if $r'=0$. Suppose $1\leq r'<r$ and $F:U\subset\mathbb P^{r,s,t}\rightarrow\mathbb P^{r',s',t'}$ is a local sign-preserving map.

For a null point $x\in U$, we have $x\in x^\perp\cap U$ and so if we take a positive point $p_1\in U$ close enough to $x$, then  $p_1^{\perp}\cap U$ is also non-empty.  Note that $p_1^{\perp}$ is an $(r-1,s,t)$-subspace $H^{r-1,s,t}$ and by the sign-preserving property $(F(p_1))^{\perp}$ is an $(r'-1,s',t')$-subspace $H^{r'-1,s',t'}$.  So the restriction of $F$ on $U_1:=H^{r-1,s,t}\cap U$ is a local sign-preserving map from $H^{r-1,s,t}$ to $H^{r'-1,s',t'}$. By choosing a convex $U$ with respect to the standard coordinates, we may assume that $U_1$ remains convex and connected. Now, as $r-1\geq 1$ and $s\geq 1$, we can always choose $p_1$ such that $U_1$ contains both positive and negative points and hence also null points.

We repeat the same procedures on $U_1$. Inductively, after taking $r'$ positive points $p_1,\ldots, p_{r'}$, we get from the restriction of $F$ a local sign-preserving map from $H^{r-r',s,t}$ to $H^{0,s',t'}$. This is a contradiction since $H^{r-r',s,t}$ contains positive points but $H^{0,s',t'}$ does not.
\end{proof}

\begin{theorem}\label{same2}
Let $F$ be a local sign-preserving map from $ \mathbb P^{r,s,t}$ to $\mathbb P^{r',s', t'}$, where $r,s\geq 2$. If $r=r'$ or $s=s'$ then $F$ is quasi-standard {\rm (}standard for $t'=0${\rm )}.
\end{theorem}

\begin{proof}
%When $s\le r$, by Theorem \ref{same}, $F$ is quasi-standard for $\mathbb P^{r,s,t}\cong \mathbb P^{s,r,t}$ and  $\mathbb P^{r', s,t'}\cong \mathbb P^{s,r',t'}$ by mapping $[z^1,z^2,z^3]$ to  $[z^2,z^1,z^3]$.
  
By symmetry, it suffices to prove the theorem for $r=r'$. We begin with the case $t'=0$.  If $r\leq s$, then everything follows from Theorem~\ref{same}. Suppose now $r>s$.

By following the same procedures of taking orthogonal complements as in the proof of Proposition~\ref{less2}, we can pick $(r-s)$ orthogonal positive points in $U$ such that the restriction of $F$ on their orthogonal complement is a local sign-preserving map from some $(s,s,t)$-subspace $H^{s,s,t}\subset\mathbb P^{r,s,t}$ to an $(s,s')$-subspace $H^{s,s'}\subset\mathbb P^{r',s'}$. 

As $s\geq 2$, by Theorem~\ref{same}, $F|_{H^{s,s,t}}$ is standard. In particular, $F|_{H^{s,s,t}}$ maps the positive $(s-1)$-planes in $H^{s,s,t}$ to positive $(s-1)$-planes in $\mathbb P^{r',s'}$. Note that our argument would give the same conclusion for any other $(s,s,t)$-subspace $\tilde H^{s,s,t}$ close enough to $H^{s,s,t}$ (as points in the Grassmannian). Since the set of positive $(s-1)$-planes is open in the Grassmannian $\mathbb G(s-1,\mathbb P^{r,s,t})$, it follows that $F$ maps $(s-1)$-planes to $(s-1)$-planes. Thus, by Proposition~\ref{faran type}$(i)$, $F$ is either linear or the image of $F$ is contained in an $(s-1)$-plane. The latter is impossible since we already know that $F|_{H^{s,s,t}}$ is standard. So $F$ is linear and together with the sign-preserving hypothesis, we conclude that $F$ is standard by Lemma~\ref{linear}.

Finally, if $t'\geq 1$, by considering again the projection from $\mathbb P^{r',s',t'}$ to $\mathbb P^{r',s'}$ as in Theorem~\ref{same}, we see that $F$ is quasi-standard. 
\end{proof}

%{\color{red} \noindent{\bf Remark:} From the theorem above and Proposition \ref{snv}, we know there is no a global sign-preserving map from $\mathbb P^{r,s,t}$ to $\mathbb P^{r',s}$ when $2\le r\le r'$, $s\ge 2$ and $t>0$.}

%Suppose $r'< r$ and let $F:U\subset\mathbb P^{r,s,t}\rightarrow\mathbb P^{r',s',t'}$ be a sign-preserving map, where $U$  is a connected open set in $\mathbb P^{r,s,t}$. Take a positive point $p_1$ close to a null point in $U$. Then $F(p_1)$ is also a positive point and $p_1^{\perp}$ is an $(r-1,s,t)$-subspace in $\mathbb P^{r,s,t}$. We take a positive point $p_2$ in $p_1^{\perp}\cap U$ close to a null point in $U$, then $F(p_2)$ is a positive point in $F(p_1)^{\perp}$. Taking a positive point $p_3$ in  $p_1^{\perp}\cap p_2^{\perp}\cap U$ close to a null point in $U$, we see that $F(p_3)$ is also a positive point in $F(p_1)^{\perp}\cap F(p_2)^{\perp}\cap U$. Do $r$ steps until we take a positive point $p_r$ in $p_1^{\perp}\cap\cdots \cap p_{r-1}^{\perp}\cap U$. Then we get a sequence of positive points $F(p_1),\cdots, F(p_r)$ satisfying $F(p_i)\in F(p_j)$,  for any $i,j\in \{1,\cdots ,r\}$, $i\neq j$. Hence the linear span of $F(p_1),\cdots, F(p_r)$ is a $(r,0)$-subspace in $\mathbb P^{r',s',t'}$. A contradiction for $r< r'$. The proof for the case $s'< s$ is similar. 
%\end{proof}

\begin{theorem}\label{double dim thm}
Let $F$ be a local orthogonal map from $\mathbb P^{r,s}$ to $\mathbb P^{r',s',t'}$. If 
$$r'+s'\leq 2\dim(\mathbb P^{r,s})-1,$$ then $F$ is either null or quasi-linear. If in addition $F$ preserves the sign of any single positive or negative point, then $F$ is quasi-standard.
\end{theorem}
\begin{proof}
The statement is trivial if $\dim(\mathbb P^{r,s})=1$. We assume $\dim(\mathbb P^{r,s})\geq 3$ for the moment and the two dimensional case will be included later in the argument.

Suppose $F$ is not null and $r'+s'\leq 2\dim(\mathbb P^{r,s})-1$.
Let $F_1:=\pi_1\circ F$, where $\pi_1:\mathbb P^{r',s',t'}\dasharrow\mathbb P^{r',s'}$ is the standard projection. If $F_1$ is not linear, then by Proposition~\ref{faran type}$(i)$, either image of $F_1$ is contained in a line or the linear span $S$ of the image of a general line $L$ under $F_1$ is of dimension $d\geq 2$. In the latter situation, since $\dim(L^\perp)=\dim(\mathbb P^{r,s})-2$ and $\dim(S^\perp)=\dim(\mathbb P^{r',s'})-d-1$ (here $S^\perp$ is the orthogonal complement of $S$ in $\mathbb P^{r',s'}$) and any $(\dim(\mathbb P^{r,s})-2)$-plane in $\mathbb P^{r,s}$ is the orthogonal complement of some line, we see by orthogonality that $F_1$ maps $(\dim(\mathbb P^{r,s})-2)$-planes to $(\dim(\mathbb P^{r',s'})-d-1)$-planes. (This in particular also implies that we must have $\dim(\mathbb P^{r',s'})\geq 3$.) 

As $\dim(\mathbb P^{r,s})-2\geq 1$ and
$$
	\dim(\mathbb P^{r',s'})-d-1\leq r'+s'-4\leq 2\dim(\mathbb P^{r,s})-5=2(\dim(\mathbb P^{r,s})-2)-1,
$$
Proposition~\ref{faran type}($ii$) implies that the linear span $H_1$ of the image of $F_1$ is of dimension at most 
$$
\dim(\mathbb P^{r',s'})-d-1+2= \dim(\mathbb P^{r',s'})-d+1\leq \dim(\mathbb P^{r',s'})-1.
$$
Therefore, we have shown that if $F_1$ is not linear, then the image of $F_1$ lies in a hyperplane of $\mathbb P^{r',s'}$. %if $\dim(\mathbb P^{r',s'})\geq 2$.

If we write $H_1\cong\mathbb P^{r_1',s_1',t_1'}$, then we have $r_1'+s_1'\leq r'+s'\leq 2\dim(\mathbb P^{r,s})-1$. So the hypotheses of the theorem still hold for the local orthogonal map $F_1$ from $\mathbb P^{r,s}$ to $H_1$ and thus by similarly considering the standard projection $\pi_2:H_1\dasharrow K_1$ for some $(r_1',s_1')$-subspace $K_1\subset H_1$ and $F_2:=\pi_2\circ F_1$, etc., we deduce inductively there exists some $(r'',s'')$-subspace $\Phi\subset\mathbb P^{r',s',t'}$ with $F_\Phi:=\pi_\Phi\circ F$, where $\pi_\Phi:\mathbb P^{r',s',t'}\dasharrow\Phi$ is the associated projection, such that we have two possibilities: 

$(i)$  $F_\Phi$ is orthogonal and linear; or

$(ii)$ $F_\Phi$ is orthogonal and there is a line $L\subset\Phi$ containing the image of $F_\Phi$.

(We note here that if $\dim(\mathbb P^{r,s})=2$ (which implies also $\dim(\mathbb P^{r',s'})\leq 2$), by replacing a general line $L$ in the argument above by a general point, we see immediately by orthogonality that $F_1$ maps lines to lines and so we still end up with the two possibilities above (with $F_\Phi$ being $F_1$).)

We are going to first show that case $(ii)$ will lead to a contradiction. If $L\cong\mathbb P^{1,0,1}$ or $\mathbb P^{0,1,1}$, then $F_\Phi$ must be constant (and hence null) since there is only one null point in $L$. However, $F$ is assumed to be not null and $F_\Phi$ is obtained from $F$ by composing a number of projections which are standard each time, thus $F_\Phi$ cannot be null neither. So we can only have $L\cong\mathbb P^{1,1}$. Consequently, $F_\Phi$ can be regarded as a (non-null) local orthogonal map from $\mathbb P^{r,s}$ to $\mathbb P^{1,1}$ and by Theorem~\ref{less}, we get that $\min\{r,s\}=1$. If $\max\{r,s\}\geq 2$, for any point $p\in\mathbb P^{r,s}$ at which $F_\Phi$ is defined, we have $\dim(p^\perp)=\max\{r,s\}-1\geq 1$. However, we also have $\dim(F_\Phi(p)^\perp)=0$ since $F_\Phi(p)\in\mathbb P^{1,1}$, so by orthogonality $F_\Phi$ maps lines to points. Therefore, $F_\Phi$ is constant and hence null, contradicting our assumption at the beginning that $F$ is not null. Therefore, we must have $\max\{r,s\}=1$ and hence $r=s=1$. This is again a contradiction since $\dim(\mathbb P^{r,s})\geq 2$.

For case $(i)$, if the image of $F$ is contained in $\Phi$, then $F$ is just $F_\Phi$ and is linear. Otherwise, consider the map $F_{\Phi^\perp}:=\pi_{\Phi^\perp}\circ F$, where $\pi_{\Phi^\perp}:\mathbb P^{r',s',t'}\dasharrow\Phi^\perp$ is the canonical projection. Then, $F_{\Phi^\perp}$ is a local orthogonal map from $\mathbb P^{r,s}$ to $\Phi^\perp\cong\mathbb P^{r'-r'',s'-s'',t'}$ by Proposition~\ref{projection 2}. If $F_{\Phi^\perp}$ is null, then $F$ is quasi-linear. If not, we can repeat the entire argument above on $F_{\Phi^\perp}$ and use induction to conclude that $F$ is quasi-linear.

Finally, if $F$ maps some positive (resp. negative) point to a positive (resp. negative) point, then the linear part of $F$ is standard by Lemma~\ref{linear} and so $F$ is quasi-standard.
\end{proof}

\begin{theorem}\label{main}
Let $F$ be a local orthogonal map from $\mathbb P^{r,s}$ to $\mathbb P^{r',s',t'}$. If $$\min\{r',s'\}\leq 2\min\{r,s\}-2,$$ then $F$ is either null or quasi-linear. If in addition $F$ preserves the sign of any single positive or negative point, then $F$ is quasi-standard.
\end{theorem}
\begin{proof}
The theorem is trivial if $\min\{r,s\}\leq 1$ since it would imply $\min\{r',s'\}=0$. Suppose $\min\{r,s\}\geq 2$. As before, we just need to prove the case $t'=0$. By Proposition \ref{boundary}, $F$ maps ${(\min\{r,s\}-1)}$-planes to $(\min\{r',s'\}-1)$-planes.
By hypotheses $\min\{r',s'\}-1\le 2(\min\{r,s\}-1)-1$, and since $\dim(\mathbb P^{r,s})=\min\{r,s\}-1+\max\{r,s\}$, we deduce from Proposition~\ref{faran type}$(ii)$ that  the image of $F$ is contained in a linear subspace $\Xi\subset\mathbb P^{r',s'}$ such that $\dim(\Xi)\leq \min\{r',s'\}-1+\max\{r,s\}$. If we write $\Xi\cong\mathbb P^{a_1,b_1,c_1}$ for some non-negative integers $a_1,b_1,c_1$, then we can regard $F$ as a local orthogonal map from $\mathbb P^{r,s}$ to $\mathbb P^{a_1,b_1,c_1}$. Note that
$$
a_1+b_1\leq \dim(\Xi)+1\leq\min\{r',s'\}+\max\{r,s\}\leq 2\min\{r,s\}-2+\max\{r,s\}.
$$ 
Thus, $a_1+b_1< 2(r+s)-3=2\dim(\mathbb P^{r,s})-1$ and now the desired results follows directly from Theorem~\ref{double dim thm}. 
\end{proof}

\begin{theorem}\label{ball}
Let $F$ be a local orthogonal map from $\mathbb P^{1,s}$ to $\mathbb P^{1,s',t'}$. If $s'\leq 2s-2$, then $F$ either null or quasi-standard. If in addition if $t'=0$, then $F$ is either constant or standard.
\end{theorem}
\begin{proof}
As usual, we just need to prove the theorem for $t'=0$. Let $F:U\subset\mathbb P^{1,s}\rightarrow \mathbb P^{1,s'}$ be a non-constant orthogonal map and $s'\leq 2s-2$. 
Take a null point $x\in U$. Since we are in $\mathbb P^{1,s}$, it follows that $x^\perp$ is a semi-negative hyperplane and $x$ is the only null point in $x^\perp$.  By orthogonality, $F(x^\perp\cap U)\subset (F(x))^\perp$ and $F(x)$ is also the only null point in the semi-negative hyerplane $F(x)^\perp\subset\mathbb P^{1,s'}$. As $F$ is non-constant, we can always choose $x$ such that $F$ is not constant on $x^\perp\cap U$ and hence $F$ preserves the sign of some negative point on $x^\perp$.

Now since $1+s'\leq 2\dim(\mathbb P^{1,s})-1$, from Theorem~\ref{double dim thm} we have that $F$ is quasi-standard. Finally, as the orthogonal complement of a $(1,s)$-subspace in $\mathbb P^{1,s'}$ is a $(0,s'-s)$-subspace which does not contain any null point, we see from the definition of quasi-standard maps that $F$ is actually standard.
\end{proof}

%{\color{red} \noindent{\bf Remark:} The global sign-reserving from $\mathbb P^{1,s,t}$ to $\mathbb P^{1,s'}$ does not exist by Proposition \ref{snv}.}

%-------------------------------------------------------------------------------------------------

\section{Proper maps between generalized balls}\label{proper maps}

On $\mathbb P^{r,s,t}$, the set of positive points $\mathbb B^{r,s,t}\subset\mathbb P^{r,s,t}$ (or $\mathbb B^{r,s}\subset\mathbb P^{r,s}$) has been called a \textit{generalized ball} in the literature since $\mathbb B^{1,s}$ is just the ordinary $s$-dimensional complex unit ball $\mathbb B^s$ embedded in $\mathbb P^s$. In additon, the boundary $\partial \mathbb B^{r,s,t}$ of $\mathbb B^{r,s,t}$ is simply the set of null points on $\mathbb P^{r,s,t}$. 

A local holomorphic map $f:U\subset\mathbb P^{r,s,t}\rightarrow\mathbb P^{r',s',t'}$, defined on a connected open set $U$ such that $U\cap\partial \mathbb B^{r,s,t}\neq\varnothing$, is called a \textit{local proper holomorphic map} from $\mathbb B^{r,s,t}$ to $\mathbb B^{r',s',t'}$ if $f(U\cap \mathbb B^{r,s,t})\subset \mathbb B^{r',s',t'}$ and $f(U\cap\partial \mathbb B^{r,s,t})\subset\partial \mathbb B^{r',s',t'}$.

The following statement is a direct consequence of Proposition~\ref{equiv1}.

\begin{proposition}\label{equiv2}
By shrinking the domain of definition if necessary, a local proper holomorphic map from $\mathbb B^{r,s,t}$ to $\mathbb B^{r',s',t'}$ is a local orthogonal map from $\mathbb P^{r,s,t}$ to $\mathbb P^{r',s',t'}$.
\end{proposition}

Consequently, our results for local orthogonal maps in the previous sections also hold for local proper holomorphic maps among generalized balls, from which we can obtain and generalize many well-known results in the literature.

\begin{theorem}
When $r, s\geq 2$, every local proper holomorphic map from $\mathbb B^{r,s,t}$ to $\mathbb B^{r,s',t'}$ is quasi-standard (standard for $t'=0$).
\end{theorem}
\begin{proof}[Proof and remarks.]
The statement follows from Theorem~\ref{same2}. Here we note that although Theorem~\ref{same2} is for local sign-preserving maps, its proof actually only assumes that positive points are mapped to positive points. 

The case for $r\leq s$ and $t=t'=0$ has been obtained by Baouendi-Huang~\cite{BH} (Theorem 1.2 therein) while the case for $r\leq s$ has been obtained by Ng-Zhu~\cite{NZ} (Theorem 1.2 therein).
\end{proof}

\begin{theorem}\label{proper double dim}
When $r'+s'\leq 2(r+s)-3$, every local proper holomorphic map from $\mathbb B^{r,s}$ to $\mathbb B^{r',s',t'}$ is quasi-standard (standard for $t'=0$). In particular, every local proper holomorphic map from $\mathbb B^s$ to $\mathbb B^{s'}$ is standard if $s'\leq 2s-2$.
\end{theorem}
\begin{proof}[Proof and remarks.] 
It follows from Theorem~\ref{double dim thm}. The special case for $r+1=r'=2$ and $t'=0$ has been proven by Xiao-Yuan~\cite{XY} (Theorem 3.2 therein). The special case for the ordinary complex unit balls is a classical theorem by Faran~\cite{faran}.
\end{proof}

\begin{theorem}
When $r\leq s$ and $r'\leq 2r-2$, every local proper holomorphic map from $\mathbb B^{r,s}\rightarrow \mathbb B^{r',s',t'}$ is quasi-standard.
\end{theorem}
\begin{proof}[Proof and remarks.]
It is a direct consequence of Theorem~\ref{main}. The case $t'=0$ has been obtained by Baouendi-Ebenfelt-Huang~\cite{BEH} (Theorem 1.1 and Corollary 1.6 therein)
\end{proof}

\noindent{\bf Acknowledgements.} The authors would like to thank Prof. Xiaojun Huang, Ming Xiao, Wanke Yin and Yuan Yuan for their comments on the first version of this article. The first author was partially supported by Institute of Marine Equipment, Shanghai Jiao Tong University. The second author was partially supported by Science and Technology Commission of Shanghai Municipality (STCSM) (No. 13dz2260400).

\end{document}